\documentclass[12 pt]{amsart}

\usepackage{amsmath}
\usepackage{amsfonts}
\usepackage{amssymb}
\usepackage{graphicx}
\usepackage[margin=3cm]{geometry}
\usepackage{hyperref}
\usepackage{enumerate}
\usepackage{mathrsfs}

\usepackage{amsthm}
\usepackage{mathtools}
\newcommand{\R}{\mathbb{R}}

\newcommand{\Z}{\mathbb{Z}}

\newcommand{\Or}{\mathcal{O}}

\DeclareMathOperator{\Diff}{Diff}
\DeclareMathOperator{\Homeo}{Homeo}
\DeclareMathOperator{\Fix}{Fix}
\DeclareMathOperator{\Per}{Per}
\DeclareMathOperator{\supp}{supp}
\DeclareMathOperator{\Int}{Int}
\DeclareMathOperator{\egr}{egr}
\DeclareMathOperator{\SL}{SL}
\DeclareMathOperator{\genus}{genus}

\newtheorem{theorem}{Theorem}[section]
\newtheorem{lemma}[theorem]{Lemma}

\newtheorem{coro-llary}[theorem]{Corollary}

\theoremstyle{definition}
\newtheorem{remark}[theorem]{Remark}

\theoremstyle{definition}
\newtheorem{question}[theorem]{Question}

\theoremstyle{definition}
\newtheorem{definition}[theorem]{Definition}

\usepackage[usenames,dvipsnames]{color}

\begin{document}

\title{Distortion for diffeomorphisms of surfaces with boundary}
\author{Kiran Parkhe}


\begin{abstract}
If $G$ is a finitely generated group with generators $\{g_1, \dots, g_s\}$, we say an infinite-order element $f \in G$ is a distortion element of $G$ provided that $\displaystyle \liminf_{n \to \infty} \frac{|f^n|}{n} = 0$, where $|f^n|$ is the word length of $f^n$ with respect to the given generators. Let $S$ be a compact orientable surface, possibly with boundary, and let $\Diff(S)_0$ denote the identity component of the group of $C^1$ diffeomorphisms of $S$. Our main result is that if $S$ has genus at least two, and $f$ is a distortion element in some finitely generated subgroup of $\Diff(S)_0$, then $\supp(\mu) \subseteq \Fix(f)$ for every $f$-invariant Borel probability measure $\mu$. Under a small additional hypothesis the same holds in lower genus.

For $\mu$ a Borel probability measure on $S$, denote the group of $C^1$ diffeomorphisms that preserve $\mu$ by $\Diff_\mu(S)$. Our main result implies that a large class of higher-rank lattices admit no homomorphisms to $\Diff_{\mu}(S)$ with infinite image. These results generalize those of Franks and Handel \cite{F&H 2} to surfaces with boundary.
\end{abstract}
\maketitle

\section{Introduction}
In this article, $S$ is a compact orientable surface with boundary $\partial S$. We denote $\Int(S) = S \setminus \partial S$. If $U \subseteq S$ is open (in the topology on $S$), we denote $\partial U = U \cap \partial S$, and $\Int(U) = U \cap \Int(S)$. $\mu$ will be a (not necessarily smooth) Borel probability measure on $S$. We denote the group of $C^1$ diffeomorphisms of $S$ (preserving $\mu$) by $\Diff(S)$ (resp. $\Diff_\mu(S)$), and we denote its identity component by $\Diff(S)_0$ (resp. $\Diff_\mu(S)_0$). The support of $\mu$ is denoted $\supp(\mu)$. The set of fixed points of $f$ is denoted $\Fix(f)$, and the set of periodic points is denoted $\Per(f)$.

\begin{definition}
If $G$ is a finitely generated group, and we choose the generating set $\{g_1, \dots, g_s\}$, then $f \in G$ is said to be a \textit{distortion element} of $G$ if $f$ has infinite order and

$$\liminf_{n \to \infty} \frac{|f^n|}{n} = 0,$$

where $|f^n|$ is the word length of $f^n$ in the generators $\{g_1, \dots, g_s\}$.
\end{definition}

See Gromov \cite{Gromov} for a good discussion with many examples.

\begin{remark}
We could have taken the liminf to be a limit; because word length is subadditive, the limit must exist, and is called the \textit{translation length} (see \cite{G&S}).

It is straightforward to see that the property of being a distortion element is independent of the finite generating set chosen. It is obvious that if $G' \supseteq G$ is also finitely generated, and $f$ is distorted in $G$, then $f$ is distorted in $G'$.

If $G$ is not finitely generated, we say that $f \in G$ is distorted in $G$ if it is distorted in some finitely generated subgroup of $G$.
\end{remark}

One reason why distortion elements are interesting is that well-known groups have them. It is easy to check that the central elements of the three-dimensional Heisenberg group are distortion elements. Lubotzky, Mozes, and Raghunathan proved that irreducible nonuniform lattices in higher-rank Lie groups have distortion elements (\cite{LMR}).

Franks and Handel \cite{F&H 2} proved that for a closed surface $S$ of genus at least two, if $f \in \Diff(S)_0$ is a distortion element and $\mu$ is an $f$-invariant Borel probability measure on $S$, then $\supp(\mu) \subseteq \Fix(f)$. They also proved that the same holds for $S = S^2$ or $T^2$, provided that $f$ has at least three or one fixed points, respectively.

Our main theorem generalizes this result to the case of surfaces with boundary:

\begin{theorem}
\label{Dist}
Let $S$ be a compact surface, possibly with boundary. Let $f \colon S \to S$ be a $C^1$ diffeomorphism which is isotopic to the identity. Suppose $f$ is distorted in $\Diff(S)_0$. Suppose that the pair $(S, f)$ satisfies the following property $(\star)$:

\noindent Either $\genus(S) \geq 2$, or
\begin{itemize}
\item If $S = S^2$, $f$ has at least three fixed points.
\item If $S = D$ is a closed disk, $f$ has at least one fixed point on $\partial D$ or at least two fixed points in $\Int(D)$.
\item If $S = A$ is a closed annulus, $f$ has at least one fixed point.
\item If $S = T^2$, $f$ has at least one fixed point.
\end{itemize}

\noindent If $\mu$ is an $f$-invariant Borel probability measure, then $\supp(\mu) \subseteq \Fix(f)$. In particular, $f$ has no non-fixed periodic points, since we can always put a finite invariant measure on a periodic orbit.
\end{theorem}

\begin{remark}
The theorem actually holds if $S = S^2$ and $f \colon S^2 \to S^2$ has only one fixed point. In that case, it is trivial, in the sense that it does not depend on $f$ being a distortion element. Let $x$ be the fixed point of $f$. If $\mu$ were an invariant probability measure not supported on $\{x\}$, we would have an action of $f$ on $S^2 \setminus \{x\}$ with an invariant measure, hence recurrent points by the Poincar\'{e} recurrence theorem, and therefore fixed points in $S^2 \setminus \{x\}$ by the Brouwer plane translation theorem (see e.g. \cite{Franks}), a contradiction. Therefore, an invariant probability measure will be supported on $\{x\}$.

Note that it is possible for a distortion element of $\Diff(S^2)_0$ to have only one fixed point. Let $f', g'$, and $h' \colon \R^2 \to \R^2$ be defined in the following way: $g'(x, y) = (x + y, y)$, $h'(x, y) = (x, y + 1)$, and $f'$ is the commutator $[g', h']$, so that $f'(x, y) = (x + 1, y)$. This is a faithful action of the three-dimensional (discrete) Heisenberg group, since $g'$ and $h'$ both commute with their commutator $f'$. Let $\phi$ be the diffeomorphism of $\R^2$ given by $\phi(x, y) = (\frac{x}{1 + y^2}, y)$; we conjugate the above action of the Heisenberg group by $\phi$, i.e. we take $f = \phi f' \phi^{-1}, g = \phi g' \phi^{-1}$, and $h = \phi h' \phi^{-1}$. The reason for conjugating is that $g'$ is not differentiable at infinity, but $f, g$, and $h$ are, with derivative the identity. So we can extend $f, g$, and $h$ to be diffeomorphisms of $S^2 = \R^2 \cup \{\infty\}$; by a slight abuse of notation we use the same names for these extensions. We have found a subgroup $\langle g, h \rangle$ of $\Diff(S^2)_0$ isomorphic to the Heisenberg group, and the commutator $f$ is a distortion element fixing only $\infty$.
\end{remark}

\begin{question}
Suppose $S = D$ is the closed disk, and $f$ has only one fixed point, in $\Int(D)$. Is it possible that there exists an $f$-invariant Borel probability measure $\mu$, with support not contained in $\Fix(f)$? In particular, is an irrational rotation of the disk distorted in $\Diff(D)_0$?

The methods of Militon \cite{Militon} on recurrent diffeomorphisms may be of use, but we do not immediately see how to apply them to surfaces with boundary.
\end{question}

\begin{remark}
Aside from the above cases, if $(S, f)$ do not satisfy $(\star)$ then there are counterexamples to the conclusion of the theorem. Calegari and Freedman \cite{C&F} have shown that an irrational rotation of $S^1$ is distorted in $\Diff^1(S^1)_0$. Foliating the closed annulus by circles, we get that an irrational rotation of the annulus $A$ is distorted in $\Diff^1(A)_0$. The same procedure works for the torus $T^2$.

Calegari and Freedman proved in the same paper that irrational rotations of $S^2$ are distorted in $\Diff^\infty(S^2)_0$, showing that the theorem does not hold for distortion elements in $\Diff^1(S^2)_0$ with two fixed points.
\end{remark}

\noindent
\textbf{Applications.} Clearly, if $S$ is a surface with boundary, and $f \in \Diff(S)_0$ is a distortion element having enough fixed points (as given by Theorem \ref{Dist}), then $f$ cannot preserve any measure with full support, such as area. Moreover, Theorems 1.5 and 1.6 and Corollaries 1.7 to 1.10 of \cite{F&H 2} continue to remain valid for surfaces with boundary, with two caveats: sometimes additional hypotheses are needed in the genus $0$ case, and there is a slight misprint in Corollary 1.9, which should say ``$\Diff_\mu(S)_0$'' instead of ``$\Diff_\mu(S)$''.

For example, we have the following facts for a compact oriented surface $S$ with nonempty boundary, equipped with a Borel probability measure $\mu$:

\begin{list}
{$\bullet$}
{\setlength\leftmargin{.15in}}
\item Let $\mathscr{G}$ be an irreducible nonuniform lattice in a semisimple real Lie group of real rank at least two. Assume that the Lie group is connected, without compact factors, and with finite center. Suppose $S \neq D$ (the closed disk). Then every homomorphism $\phi \colon \mathscr{G} \to \Diff_\mu(S)$ has finite image.
\item Suppose that $\mathscr{G}$ is a finitely generated, almost simple group that has a subgroup $\mathscr{H}$ isomorphic to the three-dimensional Heisenberg group. Then any homomorphism $\phi \colon \mathscr{G} \to \Diff_\mu(S)$ has finite image.
\item Suppose that $\supp(\mu) = S$. Suppose that $\mathscr{N} \subset \Diff_\mu(S)_0$ is nilpotent. If $S \neq D$, then $\mathscr{N}$ is abelian. 
\end{list}

In verifying these claims, one uses the Thurston stability theorem (\cite{Thurston}, Theorem 3), which remains valid for a surface with boundary. One also uses a result of Zimmer (\cite{Zimmer}, Theorem 3.14) that if $\mathscr{G}$ is a finitely generated discrete group with Kazhdan property (T), and $\phi \colon \mathscr{G} \to \Diff^1(S)$ has a finite orbit, then the image of $\phi$ is finite. This also works for surfaces with boundary; it relies on the fact that a Kazhdan group has no nontrivial homomorphism to $\R$ and has only finite image homomorphisms to $\SL(2, \R)$, together with the Thurston stability theorem.

%Also: a lot of the rotation number stuff (using it to find a periodic points, etc.) works out the same for the annulus as for the torus.  And it's good to know that if you have a f.g. group of homeomorphisms of the circle with a distortion element f that is central in the f.g. group, then f must have rational rotation number (we can rule out f being conjugate or semi-conjugate to an irrational rotation).

\begin{question}
How much of this remains true if we consider $\Homeo(S)_0$ instead of $\Diff(S)_0$?
\end{question}

\section{Normal form}

The first result we will need is that a $C^1$ diffeomorphism $f$ of a compact surface with boundary is isotopic, relative to $\Fix(f)$, to a \textit{normal form} $\phi$ for $f$.

\begin{definition}
Let $S$ be a compact surface with boundary. Let $f \colon S \to S$ be an orientation-preserving homeomorphism. Let $\phi \colon S \to S$ be another homeomorphism. We say that $\phi$ is isotopic to $f$ relative to $\Fix(f)$ if there is $F \colon S \times [0, 1] \to S$, a continuous family of homeomorphisms of $S$, such that $F(x, 0) = f(x)$ and $F(x, 1) = \phi(x)$ for all $x \in S$, and if $x \in \Fix(f)$ then $F(x, t) = x$ for all $t$.

Following ideas of Handel \cite{Handel} and Franks and Handel \cite{F&H 1}, we say that $f$ \textit{has a normal form relative to its fixed point set} if there is a finite set $R$ of disjoint simple closed curves called \textit{reducing curves} in $M = \Int(S) \setminus \Fix(f)$ and a homeomorphism $\phi$ isotopic to $f$ relative to $\Fix(f)$ such that:

\begin{enumerate}
\item $\phi$ permutes disjoint open annulus neighborhoods $A_j \subseteq M$ of the elements $\gamma_j \in R$.
\end{enumerate}

\vspace{12pt}

\noindent Let $\{S_i\}$ be the components of $S \setminus \cup A_j$, let $X_i = \Fix(f) \cap S_i$, let $M_i = S_i \setminus X_i$, and let $r_i$ be the smallest positive integer such that $\phi^{r_i}(M_i) = M_i$. Note that $r_i = 1$ if $X_i \neq \varnothing$.

\vspace{12pt}

\begin{enumerate}
\setcounter{enumi}{1}
\item If $X_i$ is infinite then $\phi|_{S_i} = id$.
\item If $X_i$ is finite then $M_i$ has negative Euler characteristic and $\phi^{r_i}|_{M_i}$ is either pseudo-Anosov or periodic. In the periodic case, $\phi^{r_i}|_{M_i}$ is an isometry of a hyperbolic structure on $M_i$.
\end{enumerate}
\end{definition}

%Note: If $f$ has a normal form relative to its fixed point set, then $R(f)$ is well defined up to isotopy relative to $\Fix(f)$. This is Lemma 6.2 of \cite{F&H 1}.

The following says that Theorem 1.2 of \cite{F&H 1} still holds for surfaces with boundary.

\begin{theorem}
\label{Normalform}
If $f$ is a $C^1$ diffeomorphism of a compact surface with boundary, $S$, then $f$ is isotopic relative to $\Fix(f)$ to a homeomorphism $\phi$ which is a normal form for $f$.
\end{theorem}

\begin{lemma}
Let $A(f)$ denote the accumulation set of $\Fix(f)$. It is possible to find a neighborhood $V$ of $A(f)$ with finitely many components such that the following hold:

\begin{enumerate}
\item{$f|_V \colon V \to S$ is isotopic to the inclusion relative to $\Fix(f) \cap V$.}
\item{If $x \in V \setminus \Fix(f)$ then the path from $f(x)$ to $x$ determined by the isotopy is contained in $S \setminus \Fix(f)$.}
\end{enumerate}
\end{lemma}

\begin{proof}
We elaborate on the ideas of Lemma 4.1 of \cite{Handel}, and show that they remain valid with boundary. Choose a metric on $S$ such that the boundary components are geodesics. Note that for any $x \in S$, there is a number $d_x$ small enough so that for $y$ within $d_x$ of $x$, there is a unique shortest geodesic from $x$ to $y$.  Since $S$ is compact, there is a number $d$ which has this property for all $x \in S$.  We may choose a neighborhood $V$ of $A(f)$ small enough so that for all $x \in V$, $d(x, f(x)) < d$.  Let $\gamma_x \colon [0, 1] \to S$ be the constant-speed parametrization of this geodesic from $x$ to $f(x)$. In this way, we get a homotopy $f_t$ between the inclusion $i \colon V \hookrightarrow S$ to $f|_V \colon V \to S$, where $f_t(x) = \gamma_x(t)$.  We claim that, if $V$ is small enough, this is actually an isotopy.

Let $a \in A(f)$.  Then there is some $w \in T_a(S)$ fixed by $Df_a$.  Therefore, since $f$ is orientation-preserving, there cannot be $v \in T_a(S)$ such that $Df_a(v) = -kv$, with $k > 0$. Note that this is also true if $a \in \partial S$, in which case $T_a(S)$ is the upper half plane. It is not hard to see that $(Df_t)_a(v) = t\cdot Df_a(v) + (1 - t)\cdot v$. This implies that $(Df_t)_a(v) \neq 0$ for all $v \neq 0 \in T_a(S)$.

%Let $\gamma_v$ be the geodesic with $\gamma_v(0) = a$ and $\gamma_v'(0) = v$. Let $\gamma_\epsilon$ be the unique shortest geodesic with $\gamma_\epsilon(0) = \gamma_v(\epsilon)$ and $\gamma_\epsilon(1) = f(\gamma_v(\epsilon))$, so $\gamma_\epsilon(t) = f_t(\gamma_v(\epsilon))$. Let $\theta(\epsilon)$ be the angle between $\gamma_\epsilon'(0)$ and $-\gamma_v'(\epsilon)$; note that as $\epsilon \to 0$ this angle approaches a nonzero limit, which is equal to the angle between $-v$ and $Df_a(v) - v$. Since geodesics locally diverge linearly, the distance between $a$ and (the local portion of) $\gamma_\epsilon$ is of class $\Theta(\epsilon)$ as $\epsilon \to 0$. For any given $t \in [0, 1]$ the distance $d(a, \gamma_\epsilon(t))$ is at least this, so the curve $\gamma_t(\epsilon) = \gamma_\epsilon(t)$ has $\gamma_t'(0) \neq 0$. But $\gamma_t'(0) = (Df_t)_a(v)$, establishing the claim.

The rest follows without change from the proof of Handel \cite{Handel}, Lemma 4.1. Note that the components of $V$ form an open cover of $A(f)$; since $A(f)$ is compact, we may find a finite subcover, so throwing out components if necessary we may assume that $V$ has finitely many components.

%Therefore, there exists $\delta'(a) > 0$ such that for all $t \in [0, 1]$ and all $x \in B_{\delta'(a)}(a)$, $(Df_t)_x$ is invertible, where $B_{\delta'(a)}(a)$ is the ball of radius $\delta'(a)$ about $a$. Then there is $\delta(a) > 0$ such that each $f_t|_{B_{\delta(a)}(a)}$ is injective. Since $A(f)$ is compact, there exists $\delta_1$ such that for any $a \in A(f)$ and $x, y \in B_{\delta_1}(a)$, $f_t(x) \neq f_t(y)$.

%On the other hand, $f_t|_{A(f)} = id$, so if we choose $\delta > 0$ small enough, then $f_t|_{B_\delta(A(f))}$ is injective. This is because for small enough $\delta$ and $x, y \in B_\delta(A(f))$, if $x$ and $y$ are both within $\delta_1$ of some $a \in A(f)$ then we're done, and otherwise $d(x, y) > \delta_1 - \delta$, so if $\delta$ was chosen small enough so that elements of $B_\delta(A(f))$ move less than $\frac{\delta_1 - \delta}{2}$ under $f$, then they also move less than $\frac{\delta_1 - \delta}{2}$ under $f_t$ for each $t$ and hence $f_t(x) \neq f_t(y)$.

\end{proof}

\begin{proof}[Proof of Theorem \ref{Normalform}]
The isotopy given by the above lemma may be extended to all of $S$, by applying the following theorem.

\begin{theorem}[Morris Hirsch \cite{Hirsch}, Theorem 1.4, p. 180]
\label{Hirsch}
Let $M$ be a manifold with boundary. Let $U \subset M$ be an open set and $A \subset U$ a compact set. Let $F \colon U \times I \to M$ be an isotopy of $U$, where $I = [0, 1]$. Let $\hat{F} \colon U \times I \to M \times I$ be defined by $\hat{F}(x, t) = (F(x, t), t)$. Suppose that $\hat{F}(U \times I) \subset M \times I$ is open.  Then there is a diffeotopy of $M$ having compact support, which agrees with $F$ on a neighborhood of $A \times I$.
\end{theorem}

In our situation, the open set will be $V$. Let the compact set $A$ be the closure of a slightly smaller neighborhood of $A(f)$, such that $V \setminus A$ does not contain any fixed points of $f$. Let the manifold $M$ be $S \setminus (\Fix(f) \cap V^c)$; note that $\Fix(f) \cap V^c$ is a finite set. The image $\hat{F}(U \times I)$ will be open, because if $x \in \partial S \cap V$ then under the isotopy given above the image of $x$ will stay in $\partial S$ (because the boundary components are geodesics), so a component of $V$ touching a boundary component will continue to touch that boundary component under the isotopy. Therefore, by Theorem \ref{Hirsch}, there is a neighborhood $W \subseteq V$ of $A$ and an isotopy $F$ of $M$ such that $F_0 = id$ and $F_1|_W = f|_W$.

Throwing out components of $W$ if necessary, we may assume that each component of $W$ contains an element of $A(f)$, so $W$ consists of finitely many components. We may further assume that the components of $W$ are smooth submanifolds; in particular, each has finitely many boundary components. Finally, we may assume that a component of $W$ contains annular neighborhoods around all boundary circles in $\partial S$ that it touches. This is for the following reason: if $\gamma$ is a circle in $\partial S$, and $C$ is a component of $S \setminus W$ intersecting $\gamma$, there will be at most finitely many fixed points of $f$ in the interior of $C$; none of them will accumulate on $\partial C$. Thus we may take a small neighborhood $N$ of $\partial C$ containing no fixed points of $f$, and there is no barrier to an isotopy to the identity on $N$, which can be extended to all of $C$ by another application of Theorem \ref{Hirsch}. Note that adding annuli around boundary circles touched by components of $W$ may have the effect of gluing components of $W$ together.

The above argument holds with $f^{-1}$ replacing $f$, so we can let $F$ be an isotopy of $M$ that goes from the identity to $f^{-1}$ on $W$. Consider the isotopy $G_t = f \circ F_t$. $G_0 = f|_M$, and $G_1$ is a map which is the identity on $W$. This extends to an isotopy of $S$ where we fill in the points of $\Fix(f) \cap V^c$. This will be an isotopy relative to $\Fix(f)$ from $f$ to a map $\phi$ with $\phi|_W = id$.

The rest is quite similar to the proof of Theorem 1.2 of \cite{F&H 1}. We look at the components of $S \setminus W$, which are compact surfaces with boundary. Note that if $C$ is a component of $S \setminus W$, part of $\partial C$ may lie in $\partial S$ and part in $\Int(S)$. Since part of $\partial C$ may lie in $\partial S$, $\phi|_{\partial C}$ need not be the identity, which is a difference from \cite{F&H 1}. But $\phi|_{\partial C}$ will be isotopic to the identity, and we may apply Theorem 1.3 of Hirsch \cite{Hirsch} to $C$. The compact submanifold will be $\partial C$, and we will let $F$ be an isotopy from $i \colon \partial C \hookrightarrow C$ to $\phi^{-1}|_{\partial C}$. By Hirsch this extends to an isotopy (which we also call) $F$ of $C$. As above, we now take the isotopy $G_t = \phi \circ F_t$; close to the boundary this goes from $\phi$ to the identity. It extends to an isotopy of all of $S$ by doing nothing outside $C$. Therefore, we may assume that $\phi|_{\partial C} = id$.

As in \cite{F&H 1}, if $C$ is a disk with at most one element of $\Fix(f)$, we can add it to $W$ and after a further isotopy assume that $\phi|_{W}$ is still the identity. If $C = A$ is an annulus disjoint from $\Fix(f)$, then after an isotopy either $\phi|_A$ is a Dehn twist or $\phi|_{W \cup A} = id$. To any other components of $S \setminus W$, we may puncture at the (finitely many) fixed points of $f$, and apply Thurston's decomposition theorem; we again call the result $\phi \colon S \to S$. Form $R$ by taking any reducing curves coming from the Thurston decomposition of the components of $S \setminus W$, together with core curves of Dehn twist annuli, together with boundary curves of components of $W$ not accounted for as core curves of Dehn twist annuli (except for components of $\partial S$).

After modifying $\phi$ on tubular neighborhoods of the reducing curves, we may assume that (1) is satisfied.  Properties (2) and (3) are immediate from the construction.
\end{proof}

Note: Since $f$ is isotopic to the identity, $\phi$ will actually not permute the complementary components $S_i$ given in the definition of the normal form; it will leave them invariant. This is a conclusion of Lemma 6.3 of \cite{F&H 1} which remains valid when $S$ has boundary.

\section{An important lemma}

We need a couple of definitions, only slightly modified from Franks and Handel \cite{F&H 2}.

\begin{definition}
Let $S$ be endowed with a Riemannian metric. A smooth curve $\gamma$ has a well-defined length $\ell_S(\gamma)$. If $\tau$ is a smooth closed curve, define the \textit{exponential growth rate} of $\tau$ with respect to $f$ by $$\egr(f, \tau) = \liminf_{n \to \infty} \frac{\log(\ell_S(f^n(\tau)))}{n}.$$ Note that the exponential growth rate will be independent of the metric chosen on $S$.
\end{definition}

\begin{definition}
\label{LinearTracing}
Suppose that $f \in \Diff(S)$, that $M$ is a component of $S \setminus \partial S \setminus \Fix(f)$ with negative Euler characteristic, that $h = f|_M$ is isotopic to the identity, and that $\beta$ is an essential closed curve in $M$. If $\beta$ is peripheral in $M$, assume that the end that it encloses is blown up to a boundary component. Choose a covering translation $T \colon \tilde{M} \to \tilde{M}$ whose axis Ax$_T$ projects to a simple closed curve that is isotopic to $\beta$. Identify Ax$_T$ with $\R$ so that the action of $T$ on Ax$_T$ corresponds to a unit translation of $\R$. Let $\tilde{p} \colon \tilde{M} \to \R$ be a $T$-equivariant projection of $\tilde{M}$ onto Ax$_T$ (e.g. orthogonal projection) followed by the identification of Ax$_T$ with $\R$. Let $\tilde{h} \colon \tilde{M} \to \tilde{M}$ be the identity lift of $h$. We say that $x \in M$ \textit{linearly traces} $\beta$ if there is a lift $\tilde{x} \in \tilde{M}$ such that $$\liminf_{n \to \infty} \frac{|\tilde{p}(\tilde{h}(\tilde{x})) - \tilde{p}(\tilde{x})|}{n} > 0.$$
\end{definition}

Note that any two $T$-equivariant projections of $\tilde{M}$ onto Ax$_T$ differ by a bounded amount, so Definition \ref{LinearTracing} is independent of the choice of projection.

\begin{definition}
\label{PrimeEnds}
%Let $U$ be a connected surface without boundary, with $H_1(U, \Z/2\Z)$ finite. There exists a (essentially unique) surface $U^*$, without boundary, such that $U$ is a topological subspace of $U^*$, and $U^* \setminus U$ is finite. The points in $U^* \setminus U$ are called ``ends'' of $U$.

%Suppose $U$ is contained in a connected surface without boundary, $S$, such that the closure of $U$ in $S$ is compact. There is a natural compactification of $U$, called the \textit{prime end compactification}, which we may denote $\hat{U}$, associated with this embedding of $U$ in $S$. A component $C$ of $S \setminus U$ will touch one or more ends of $U$. If $C$ consists of more than one point, the end(s) it touches will get circles added to them in $\bar{U}$. If $C$ is just a single point, the end it touches will get a single point added to it in $\bar{U}$. For more details see \cite{Mather}. We modify this slightly: in the latter case we perform a radial ``blow up'', so our compactification adds a circle even if $C$ was only a single point. We denote our compactification $\bar{U}$.

%One reason this is important is that if $f \colon S \to S$ is a homeomorphism such that $f(U) = U$, then there is a natural way of defining an extension of $f|_U$ to a homeomorphism $\bar{f} \colon \bar{U} \to \bar{U}$, called the \textit{canonical extension}. This extension respects the operation of composition: if we have $f, g \colon S \to S$ such that $f(U) = U$ and $g(U) = U$, then $\overline{f \circ g} = \bar{f} \circ \bar{g}$.

Let $U$ be a connected surface without boundary, with $H_1(U, \Z/2\Z)$ finite. There exists a (essentially unique) surface $U^*$, without boundary, such that $U$ is a topological subspace of $U^*$, and $U^* \setminus U$ is finite. The points in $U^* \setminus U$ are called ``ends'' of $U$.

Suppose $U$ is contained in a connected surface without boundary, $S$, such that the closure of $U$ in $S$ is compact. We form a compactification $\bar{U}$ of $U$ associated with this embedding of $U$ in $S$, which is a small modification of prime end compactification: instead of adding a single point to an end whose frontier consists of a single point, we perform a radial blow up. Given a $C^1$ diffeomorphism $f \colon S \to S$ such that $f(U) = U$, there is an extension of $f|_U$ to a homeomorphism $\bar{f} \colon \bar{U} \to \bar{U}$. We call this the \textit{canonical extension} of $f|_U$ to $\bar{U}$. It respects the operation of composition: if $g$ is another $C^1$ diffeomorphism of $S$ sending $U$ to $U$, then $\overline{f \circ g} = \bar{f} \circ \bar{g}$. (We get a similar composition-respecting extension to the prime end compactification; in that case, we may deal with homeomorphisms rather than diffeomorphisms.) For details, see \cite{F&H 1}, Lemma 5.1.
\end{definition}

The following is the primary ingredient in proving the main result. It is a minor modification of Lemma 4.2 of \cite{F&H 2}; parts (4) and (5) read slightly differently and part (3) is new.

\begin{theorem}
\label{SixPoss}
Suppose that $f \in \Diff_\mu(S)_0$ has infinite order, and $\supp(\mu) \nsubseteq \Fix(f)$. Suppose that $(S, f)$ satisfies property $(\star)$. Then, after possibly replacing $f$ with an iterate, at least one of the following holds:

\begin{enumerate}
\item There is a closed curve $\tau$ such that $\egr(f, \tau) > 0$.
\item $f$ is isotopic relative to $\Fix(f)$ to a composition of nontrivial Dehn twists about a finite collection of nonperipheral, nonparallel, disjoint simple closed curves in $S \setminus \Fix(f)$.
\end{enumerate}

\vspace{12pt}

For the following, $f$ is isotopic to the identity relative to $\Fix(f)$.

\vspace{12pt}

\begin{enumerate}
\setcounter{enumi}{2}
\item There is a component $C$ of $\partial S$ such that $f|_C$ has irrational rotation number.
\item There is an $f$-invariant annular component $U$ of $S \setminus \Fix(f)$ such that the restriction of $f$ to $\Int(U) = U \setminus \partial S$ has canonical extension to $\bar{f} \colon \overline{\Int(U)} \to \overline{\Int(U)}$, and there is $x \in \Int(U)$ such that the rotation number of $x$ with respect to the lift of $\bar{f}$ which fixes points on the boundary is nonzero.
\item There exists a component $M$ of $S \setminus \partial S \setminus \Fix(f)$ with negative Euler characteristic, a simple closed curve $\beta \subseteq M$ that is essential in $M$, and $x \in M$ that linearly traces $\beta$ in $M$.
\item There exists $x \in S$ and a lift $\tilde{f}$ such that $\Fix(\tilde{f}) \neq \varnothing$ and such that the rotation vector $\rho(x, \tilde{f})$ is not zero. In this case, $S$ has genus at least one, and $\tilde{f}$ is the identity lift if $S \neq T^2$.
\end{enumerate}
\end{theorem}

\begin{proof}
We apply Theorem \ref{Normalform}. If $\phi|_{S_i}$ is pseudo-Anosov for some Thurston component $S_i$, then there will be a closed curve $\tau \subseteq S_i$ with $\egr(\phi, \tau) > 0$, hence $\egr(f, \tau) > 0$, and we have (1).

If this does not hold, then letting $\phi_k$ be a normal form for $f^k$ relative to $\Fix(f^k)$ for each $k$, no $\phi_k$ will have a component on which it is pseudo-Anosov. Either there exists a $k$ such that $\phi_k$ is a nontrivial composition of Dehn twists about simple closed curves, so after taking a finite power of $f$ we have (2), or for every $k$, $f^k$ is isotopic to the identity relative to $\Fix(f^k)$.

We must show that, in this latter case, assuming (3) does not hold, at least one of (4), (5), or (6) must hold. By Brown and Kister \cite{B&K}, since $f$ is orientation-preserving, every component of $S \setminus \Fix(f)$ is invariant under $f$. Brown and Kister state their result for manifolds without boundary, but it still works here: if $f \colon M \to M$ is a homeomorphism of a manifold with boundary, we can simply apply their result to the restriction $f|_{\Int(M)}$.

Since $\supp(\mu) \nsubseteq \Fix(f)$, there exists a component $U$ of $S \setminus \Fix(f)$ with $\mu(U) > 0$. By Poincar\'e recurrence, there are recurrent points in $U$. Suppose $U$ is homeomorphic to a closed disk, possibly with part of its boundary removed. A priori this is possible, because a component of $S \setminus \Fix(f)$ could contain parts of $\partial S$. $\partial U$ ($= U \cap \partial S$) cannot be a whole circle, or else $f$ would be a diffeomorphism of the closed disk with no fixed points. But if $x \in \partial U$ then there is a whole (open) interval $I$ in $\partial U$ containing $x$, since $U \subseteq S$ is open. Because the boundary points of $I$ in $\partial S$ are fixed, $f(I) = I$, and all points in $I$ must be sent in one direction, so they are not recurrent. Therefore, there must be a recurrent point in the interior of $U$. But by the Brouwer plane translation theorem this implies $f$ has a fixed point in the interior of $U$, a contradiction.

Now suppose $U$ is an annulus. $U$ may be homeomorphic to an open annulus, a closed annulus, or something intermediate. The possibilities are as follows:

\begin{enumerate}[(i)]
\item Both ends of $U$ are open.
\item Exactly one end of $U$ is open.
\begin{enumerate}
\item The open end has frontier consisting of a single point
\item The open end has frontier consisting of more than one point
\end{enumerate}
\item Neither end of $U$ is open.
\end{enumerate}

We claim that $\bar{f} \colon \overline{\Int(U)} \to \overline{\Int(U)}$ has a fixed point on one of the two boundary circles of $\overline{\Int(U)}$. For case (i) this is proved by Franks and Handel in their proof of Lemma 4.2 of \cite{F&H 2}. Case (ii)(b) is identical; we have an open end whose frontier consists of more than a single point, so by prime end theory the circle of prime ends corresponding to that end is fixed pointwise.

In case (ii)(a), notice that the non-open end of $U$ cannot be completely closed. If it were, it would consist of a single boundary component $C$ of $\partial S$, and $C$ would be peripheral in $S \setminus \Fix(f)$. In this case, $S$ is a closed disk, and $f$ has a single interior fixed point, which is ruled out by the assumption that $(S, f)$ satisfies property $(\star)$. Similarly, in case (iii), the ends cannot both be closed, because then $S$ would be forced to be a closed annulus and $f$ has no fixed points, which also violates the hypotheses of the theorem.

Suppose an end $e$ of $U$ is partly open and partly closed (the closed parts are arcs contained in $\partial S$). We form the prime end compactification $\overline{\Int(U)}$, and we have $\bar{f} \colon \overline{\Int(U)} \to \overline{\Int(U)}$. Maximal arcs (connected components) in $U \cap \partial S$ have exactly corresponding arcs in $\partial \overline{\Int(U)}$, and the dynamics of $f$ and $\bar{f}$ on these arcs are conjugate. Let the arc $\alpha$ be a connected component of $U \cap \partial S$, with endpoints $x_1$ and $x_2$. Without loss of generality, for any $x \in \alpha$, $f^n(x) \to x_1$ and $f^{-n}(x) \to x_2$ as $n \to \infty$. By the conjugacy, if $\bar{\alpha}$ denotes the corresponding arc in $\partial \overline{\Int(U)}$, the endpoints of $\bar{\alpha}$ are fixed by $\bar{f}$. Thus, in all cases, we have established the claim.

By the reasoning of the above paragraph, any points in $\partial U$ are wandering. Thus no measure can be supported on them, so $\mu(\Int(U)) > 0$. Therefore, by Lemma 2.2 of \cite{F&H 2}, for every lift $\tilde{\bar{f}}$ of $\bar{f}$ to the universal cover of $\overline{\Int(U)}$ there is a point $x \in \Int(U)$ such that $\rho(x, \tilde{\bar{f}}) \neq 0$. In particular, this is true for a lift of $\bar{f}$ fixing points on the boundary of the universal cover of $\overline{\Int(U)}$. This implies that (4) holds.

We are finished exploring the possibilities if $U$ is an annulus. Therefore, we may assume that $U$ has negative Euler characteristic. We claim that $\Per(f|_U) = \Fix(f|_U)$. Since (3) does not hold, after taking a finite power if necessary we may assume that $f$ has a fixed point, and no non-fixed periodic points, on each component of $\partial S$. There cannot be a non-fixed periodic point in the interior of $U$, as we can see by applying Lemma 3.7 of \cite{F&H 3} to the restriction of $f$ to the interior of $U$.

Since $f$ has fixed points on each component of $\partial S$, any boundary arcs in $U$ consist of wandering points and must have measure $0$, so $\Int(U)$ has positive measure. Analogously to \cite{F&H 2}, we apply the proof of Theorem 1.1 of \cite{F&H 1} to $M \coloneqq \Int(U)$, as follows.

We use the notation and terminology of \cite{F&H 1}. If there exists a simple closed geodesic $\gamma$ such that, for a positive measure set $P \subseteq M$, for any $x \in P$ we have $\gamma(x) = \gamma$, then (5) holds. If there is an isolated puncture $c$ in $M$ such that for all $x$ in a positive measure set $P \subseteq M$ we have that $\Or(x)$ rotates about $c$, then again (5) holds.

Suppose there exists a simple closed geodesic $\alpha \subseteq M$ such that for a positive measure set $P \subseteq M$, for all $x \in P$, $x$ is birecurrent and $\gamma(x)$ crosses $\alpha$, and for all $x$ such that $\gamma(x)$ crosses $\alpha$ the crossing is always in the same direction (so $\alpha$ is a ``partial cross section''). In this case, note that $S$ must have genus at least one, since by \cite{F&H 1}, Lemma 11.6(2), for all $x \in P$, $\gamma(x)$ is birecurrent, but if $S$ had genus zero $\alpha$ would separate it into two components, and since $\gamma(x)$ is only allowed to cross $\alpha$ in one direction, after crossing it would have to stay in one component forever. In this case, the reasoning of Theorem 1.1 implies that $\rho_\mu(f_t) \wedge [\alpha] \neq 0$, so $\rho_\mu(f_t) \neq 0$, where $f_t$ is an isotopy of $f$ to the identity relative to $\Fix(f)$, implying (6).
\end{proof}

\section{Proof of main result}

We will need the following definitions, slightly modified from Franks and Handel \cite{F&H 2}.

\begin{definition}
We define \textit{linear displacement} as follows. Let $S$ be a surface of nonpositive Euler characteristic. If $S$ is closed, we follow \cite{F&H 2}. Suppose $S$ has nonempty boundary. %Suppose, for the moment, that $S$ is not the closed annulus, so its Euler characteristic is actually negative. In this case, we can choose a hyperbolic metric on $S$; we can make it so that the boundary components are geodesics. We can glue on hyperbolic pairs of pants to form a closed hyperbolic surface, and in this way we can realize the universal cover $\tilde{S}$ as a subset of the Poincar\'{e} disk $H$. For any $f \in \Homeo(S)_0$, there is certainly a lift $\tilde{f} \colon \tilde{S} \to \tilde{S}$ fixing the points on $S_\infty$ in the boundary of $\tilde{S}$. In fact, such a lift is unique, since if $\tilde{f}'$ were another such lift then $\tilde{f}' \circ \tilde{f}^{-1}$ would be a covering translation fixing infinitely many points on $S_\infty$, hence the identity. So in this context we still have identity lifts.

Let $d$ be a metric on $S$, and $\tilde{d}$ lift of $d$ to the universal cover $\tilde{S}$. We will say that $f$ has \textit{linear displacement} if either of the following conditions holds:

\begin{itemize}
\item $S \neq A$, $\tilde{f}$ is the identity lift, and there exists $\tilde{x} \in \tilde{S}$ such that $$\liminf_{n \to \infty} \frac{\tilde{d}(\tilde{f}^n(\tilde{x}), \tilde{x})}{n} > 0.$$

\item $S = A$, and there exists a lift $\tilde{f}$ and $x_1, x_2 \in \tilde{A} = \R \times [0, 1]$ such that $$\liminf_{n \to \infty} \frac{\tilde{d}(\tilde{f}^n(\tilde{x}_1), \tilde{f}^n(\tilde{x}_2))}{n} > 0.$$
\end{itemize}
\end{definition}

\begin{definition}
We define \textit{spread} in the following way. Suppose $S$ is a compact surface (possibly) with boundary. Let $\gamma$ be an embedded smooth path. Suppose that one of the following holds: the endpoints of $\gamma$ lie in distinct components of $\partial S$; one endpoint of $\gamma$ is in $\partial S$ and the other lies in $\Fix(f)$; or the endpoints of $\gamma$ are distinct elements of $\Fix(f)$. Let $\beta$ be a simple closed curve lying in $S$ that crosses $\gamma$ exactly once. Let $A$ be the endpoint set of $\gamma$. If an endpoint of $\gamma$ lies in $\Int(S)$, remove it and blow up the puncture to a boundary circle; if it lies in $\partial S$, do nothing. Call the resulting surface $M$. We can think of $\gamma$ and $\beta$ as curves in $M$. Suppose for the moment that $S$ has genus at least one; then $M$ has negative Euler characteristic, and we choose a hyperbolic structure on $M$. Choose nondisjoint lifts $\tilde{\beta}, \tilde{\gamma} \subseteq \tilde{M}$, and let $T \colon \tilde{M} \to \tilde{M}$ be the covering translation corresponding to $\tilde{\beta}$.  Denote $T^i(\tilde{\gamma})$ by $\tilde{\gamma}_i$. Each $\tilde{\gamma}_i$ is an embedded path in $\tilde{M}$. Moreover, $\tilde{\gamma}_i$ separates $\tilde{\gamma}_{i - 1}$ from $\tilde{\gamma}_{i + 1}$.

An embedded smooth path $\alpha \subseteq S$ whose interior is disjoint from $A \cap \Int(S)$ can be thought of as a path in $M$. For each lift $\tilde{\alpha} \subseteq \tilde{M}$, there exist integers $a < b$ such that $\tilde{\alpha} \cap \tilde{\gamma}_i \neq \varnothing$ if and only if $a < i < b$. Define $$\tilde{L}_{\tilde{\beta}, \tilde{\gamma}}(\tilde{\alpha}) = \max\{0, b - a - 2\}$$ and $$L_{\beta, \gamma}(\alpha) = \max\{\tilde{L}_{\tilde{\beta}, \tilde{\gamma}}(\tilde{\alpha})\}$$ as $\tilde{\alpha}$ varies over all lifts of $\alpha$.

Now suppose that the Euler characteristic of $M$ is not negative, so it is $0$. Thus $M$ is a closed annulus. In this case, $\tilde{M}$ is identified with $\R \times [0, 1]$, where $T(x, y) = (x + 1, y)$, and $\tilde{\gamma}$ goes from one boundary component to the other. With these modifications, $L_{\beta, \gamma}(\alpha)$ is defined as in the previous case.

There is an equivalent definition of $L_{\beta, \gamma}(\alpha)$ which does not involve covers. Namely, $L_{\beta, \gamma}(\alpha)$ is the maximum value $k$ for which there exist subarcs $\gamma_0 \subseteq \gamma$ and $\alpha_0 \subseteq \alpha$ such that $\gamma_0\alpha_0$ is a closed path that is freely homotopic to $\beta^k$ relative to $A \cap \Int(S)$. We allow the possibility that $\gamma$ and $\alpha$ share one or both endpoints. The finiteness of $L_{\beta, \gamma}(\alpha)$ follows from the smoothness of the arcs $\alpha$ and $\gamma$.

Define the \textit{spread} of $\alpha$ with respect to $f, \beta$, and $\gamma$ to be $$\sigma_{f, \beta, \gamma}(\alpha) = \liminf_{n \to \infty} \frac{L_{\beta, \gamma}(f^n \circ \alpha)}{n}.$$
\end{definition}

\begin{lemma}
\label{FourPossibilities}
Suppose that $f \in \Diff_\mu(S)_0$ has infinite order and that $\supp(\mu) \nsubseteq \Fix(f)$. Suppose that $(S, f)$ satisfies property $(\star)$. Then, after possibly replacing $f$ with an iterate, at least one of the following holds:

\begin{enumerate}
\item There is a closed curve $\tau$ such that $\egr(f, \tau) > 0$.
\item $f$ has linear displacement.
\item There is a $k$-fold cover $S_k$ of $S$ with $k = 1$ or $2$ and a lift $f_k \colon S_k \to S_k$ of $f \colon S \to S$ which is isotopic to the identity, and $\alpha, \beta,$ and $\gamma$ exist as in the definition of spread, such that $\sigma_{f_k, \beta, \gamma}(\alpha) > 0$.
\item There is a component $C$ of $\partial S$ such that $f|_C$ has irrational rotation number.
\end{enumerate}
\end{lemma}

\begin{proof}
We closely follow the proof of Corollary 5.5 of \cite{F&H 2}. We may assume that $\partial S \neq \varnothing$, since the result is proved in \cite{F&H 2} for closed surfaces. It's obvious that Theorem \ref{SixPoss}(1) implies Lemma \ref{FourPossibilities}(1) and Theorem \ref{SixPoss}(3) implies Lemma \ref{FourPossibilities}(4). If $S$ is a surface (of genus at least one) with nonempty boundary and Theorem \ref{SixPoss}(6) holds, some point has nonzero rotation vector with respect to the identity lift, and Lemma \ref{FourPossibilities}(2) follows by definition of linear displacement.

Suppose Theorem \ref{SixPoss}(4) holds. As in \cite{F&H 2}, Corollary 5.5, we let $\beta$ be the core curve of $U$. We let $\gamma$ be an arc with interior in $U$ which extends to an arc $\bar{\gamma}$ in $\bar{U}$ such that the endpoints of $\gamma$ in $S$ are in $\Fix(f)$; this is possible even if one end of $U$ consists entirely of a component of $\partial S$ since that component will have a fixed point of $f$. We let $\alpha$ be an arc with interior in $U$ that has $x$ as one endpoint (where $x$ is a point with positive rotation number, as given in Theorem \ref{SixPoss}) and extends to an arc $\bar{\alpha}$ with endpoint in $\partial \bar{U}$ a fixed point of $\bar{f}$. Then $\sigma_{f, \beta, \gamma}(\alpha) > 0$, implying Lemma \ref{FourPossibilities}(3) with $k = 1$.

Suppose Theorem \ref{SixPoss}(5) holds, so there is a hyperbolic component $M$ of $S \setminus \partial S \setminus \Fix(f)$, a point $x \in M$, and a simple closed curve $\beta \subseteq M$ such that $x$ linearly traces $\beta$ in $M$. As in \cite{F&H 2}, Corollary 5.5, we may pass to a twofold cover of $S$ if necessary to assume that there exists a smooth curve $\gamma$ with interior in $M$, crossing $\beta$ exactly once, with endpoints either in $\Fix(f) \cap \Int(S)$ or in $\partial S$. Then we apply the reasoning of Lemma 5.3 of \cite{F&H 2} without change to conclude that $\sigma_{f, \beta, \gamma}(\gamma) > 0$, so Lemma \ref{FourPossibilities}(3) holds with $k = 1$ or $2$.

Finally, suppose Theorem \ref{SixPoss}(2) holds: (Some power of) $f$ is isotopic, relative to $\Fix(f)$, to a composition of Dehn twists about a collection $R(f)$ of disjoint, nonparallel (in $S \setminus \Fix(f)$), nonperipheral simple closed curves. Let $R_e(f) \subseteq R(f)$ be those curves that are essential in $S$.

Suppose $R_e \neq \varnothing$. Then $S$ cannot be the closed disk. If $S$ has genus $0$, it must have at least $2$ boundary components. If $S = A$ is the closed annulus, we can find fixed points --- possibly on $\partial A$; note that we are assuming $f$ has fixed points on each boundary component --- $x_1$ and $x_2$ and a path $\tau$ connecting them such that $\tau$ has exactly one intersection with $R_e$. Then no lift of $f$ to the infinite strip $\tilde{A}$ fixes the full preimage of $x_1$ and $x_2$, so we have Lemma \ref{FourPossibilities}(2). If $S \neq A$, then we can make use of the identity lift $\tilde{f} \colon \tilde{S} \to \tilde{S}$. We can find a point $x \in \Fix(f)$, a lift $\tilde{x} \in \tilde{S}$, and a ray $\tilde{\sigma}$ connecting $\tilde{x}$ to a point in $S_\infty$ such that $\tilde{\sigma}$ crosses exactly one element of the full preimage $\tilde{R}_e$ of $R_e$. In this case, $\tilde{x}$ is not fixed by the identity lift $\tilde{f}$, so again Lemma \ref{FourPossibilities}(2) holds.

Now suppose that $R_e = \varnothing$, so every element of $R(f)$ bounds a disk in $S$ which contains at least two points in $\Fix(f)$. After passing to a twofold cover if necessary, we can assume it is not the case that all elements of $\Fix(f)$ lie in the same disk in $S \setminus R(f)$. Therefore, choose $x_1$ and $x_2$ lying in different disks. Exactly as in \cite{F&H 2}, we can find a curve $\gamma$ connecting them such that $\sigma_{f, \beta, \gamma}(\gamma) > 0$, so Lemma \ref{FourPossibilities}(3) holds.
\end{proof}

\begin{lemma}
\label{IrrRot}
Let $S$ and $f \colon S \to S$ be as in the statement of Theorem \ref{Dist}. If item (4) of Lemma \ref{FourPossibilities} holds, then $f$ is not a distortion element in $\Diff(S)_0$.
\end{lemma}

\begin{proof}
First suppose $S$ has genus $0$. The number of boundary components plus the number of fixed points of $f$ is at least $3$. After blowing up fixed points or collapsing boundaries to points if necessary, we may assume that in addition to $C$ there is exactly one other boundary, $C'$, so $S = A$ is a closed annulus. There will be at least one fixed point, call it $p$.

Let $\tilde{A}$ be the universal cover of $A$. Identify the boundary components of $\tilde{A}$ with $\R$ in such a way that translation by 1 corresponds to making a full circle around $C$ or $C'$. Choose a lift $\tilde{p}$ of $p$.

If $g \colon A \to A$ is an orientiation-preserving homeomorphism preserving the boundaries, then we may lift it to a homeomorphism of $\tilde{A}$ (in fact, there are countably many lifts, parametrized by $\Z$). If $g(p) = p$, there is a canonical lift $\tilde{g}$ such that $\tilde{g}(\tilde{p}) = \tilde{p}$. This lifting preserves the group structure: if $g$ and $h$ are two homeomorphisms fixing $p$, then $\widetilde{h \circ g} = \tilde{h} \circ \tilde{g}$. Let us define $\hat{g} \colon \R \to \R$ to be the restriction of $\tilde{g}$ to $\tilde{C}$ (under its identification with $\R$), so $\widehat{h \circ g} = \hat{h} \circ \hat{g}$.

Suppose $f$ is distorted in $\Diff(S)_0$; let $f_1, \dots, f_s \in \Diff(S)_0$, such that $f \in \langle f_1, \dots, f_s \rangle$, and such that $\displaystyle \liminf_{n \to \infty} \frac{|f^n|}{n} = 0$, where $|\cdot|$ is the word length in the generators $f_1, \dots, f_s$.

There is a number $M$ such that for all $x \in \R$ and $i = 1, \dots, s$, $|\hat{f_i}(x) - x| < M$. This implies that $\displaystyle \liminf_{n \to \infty} \frac{|\hat{f}^n(0) - 0|}{n} \leq \liminf_{n \to \infty} \frac{M|f^n|}{n} = 0$. But $\hat{f}$ is a lift of an irrational rotation, so it has irrational translation number, a contradiction.

The idea similar if the genus of $S$ is $\geq 1$. Construct $S_C$ from $S$ by collapsing all boundaries of $S$ except $C$ to points. Form $S'$ from $S_C$ by filling in the hole bounded by $C$ with a disk. The universal cover $\tilde{S'}$ is (homeomorphic to) an open disk. Let $\tilde{S_C} \subseteq \tilde{S'}$ be the preimage of $S_C$ (not the universal cover of $S_C$).

Let $g \colon S \to S$ be a homeomorphism isotopic to the identity. This yields a homeomorphism $\hat{g}$ of $S_C$. We can choose a lift $\tilde{g}$ of $\hat{g}$ to $\tilde{S_C}$ that preserves a specified component $\tilde{C}$ of the preimage of $C$, and since $g$ is isotopic to the identity $\tilde{g}$ will preserve all components of the preimage of $C$. If $h \colon S \to S$ is also isotopic to the identity, then $\widetilde{h \circ g} = \tilde{h} \circ \tilde{g}$, since $\widetilde{h \circ g}$ and $\tilde{h} \circ \tilde{g}$ differ by at most a covering translation that fixes preimages of $C$.

Finally, we can form an annulus $A$ by collapsing all preimages of $C$ except $\tilde{C}$ to points. Let $p$ be one of these points coming from collapsing preimages of $C$. Then $\tilde{g}$ induces a homeomorphism $\widehat{\tilde{g}}$ of $A$ which fixes $p$. We are now in exactly the situation above, since all the operations that led from $g$ to $\widehat{\tilde{g}}$ preserved group structure.
\end{proof}

\begin{proof}[Proof of Theorem \ref{Dist}]
We must show that each of the four items of Lemma \ref{FourPossibilities} will imply that $f$ cannot be a distortion element. If item (1) holds, then non-distortion follows exactly as in Lemma 6.3 of \cite{F&H 2}. If item (2) holds, then whether $S = A$ or $S \neq A$ non-distortion follows exactly according to the reasoning of Lemma 6.1 of \cite{F&H 2}. If item (4) holds, then we have non-distortion by Lemma \ref{IrrRot}.

Suppose item (3), positive spread, holds. Then the reasoning of Lemmas 6.6 and 6.8 of \cite{F&H 2} remains valid, and we can slightly modify the reasoning of Lemma 6.7 of \cite{F&H 2}. Let $\gamma$ be as in our definition of spread; in particular, its endpoints are either fixed by $f$ or lie in boundary components of $S$.

If the endpoints $x$ and $y$ of $\gamma$ are both in $\Fix(f) \cap \Int(S)$, then the reasoning of Lemma 6.7 goes through without change, but note that $g_{x',y'}$ is only defined when $x', y' \in \Int(S)$ (there is no diffeomorphism of $S$ moving $x$ and $y$ to $x'$ and $y'$ if one of $\{x', y'\}$ lies in $\partial S$). If $x \in \Fix(f) \cap \Int(S)$ and $y \in \partial S$, then we can do the same procedure, except we only look at $x', y'$ such that $y'$ lies on the same component of $\partial S$ as $y$. In either case, Lemma 6.8 continues to hold, so $f$ is not distorted.

If both $x$ and $y$ are in different components of $\partial S$, then it is easier. In that case, if $f$ were a distortion element in some finitely generated subgroup of $\Diff(S)_0$ --- say the group generated by $f_1, \dots, f_n$ --- then the $f_i$ will leave invariant the boundary components of $S$. The analogue of Lemma 6.7 holds for the $f_i$, namely, there exists $K(f_i)$ such that if $\alpha$ is a curve as in the definition of spread, then $$L_{\beta, \gamma}(f_i(\alpha)) \leq L_{\beta, \gamma}(\alpha) + K(f_i).$$ This is true because there is a uniform bound on the distance points in the universal cover move under any lift $\tilde{f_i} \colon \tilde{S} \to \tilde{S}$ to the universal cover of $S$. Thus we again conclude that $f$ is not distorted.
\end{proof}

%\section{Applications}
%Most of the applications in Franks and Handel \cite{F&H 2} remain valid in the context of sufaces with boundary. We state the following results without detailed %proofs, referring the reader to \cite{F&H 2}.
%
%\begin{theorem}
%Suppose that $S$ is a closed oriented surface of positive genus with boundary, equipped with a Borel probability measure $\mu$, and suppose that $\mathscr{G}$ is a finitely generated group that is almost simple and possesses a distortion element $u$. Suppose further that either $\mu$ has infinite support or that $\mathscr{G}$ is a Kazhdan group. Then any homomorphism $\phi \colon \mathscr{G} \to \Diff_\mu(S)$ has finite image. %The result is valid for S = S2 with the additional assumption that
%ƒÓ(u) has at least three periodic points.
%\end{theorem}

\end{document}